\documentclass[10pt]{amsart}

\usepackage{hyperref}
\hypersetup{ colorlinks=true, linkcolor=blue, filecolor=magenta, urlcolor=cyan, }

\usepackage{hyperref}
\hypersetup{ colorlinks=true, linkcolor=blue, filecolor=magenta, urlcolor=cyan, }

\usepackage{graphics}
\usepackage{graphicx}
\usepackage{epsfig}
\usepackage{amsmath}
\usepackage{graphics}
\usepackage{graphicx}
\usepackage{epstopdf}
\usepackage{multirow}
\usepackage{booktabs}
\usepackage{ifpdf}
\usepackage{color}
\usepackage{cite}
\usepackage{amssymb}
\usepackage{amsthm}
\usepackage{amsmath}
\usepackage{matlab-prettifier}
\usepackage{tikz-cd,cleveref}
\usepackage{todonotes}
\usepackage{nicematrix}
\usepackage{ytableau}
\usepackage{tikz}
\usetikzlibrary{decorations.pathreplacing,
	calligraphy,% must be loaded after decorations.pathreplacing
	matrix}

\theoremstyle{plain}% Theorem-like structures provided by amsthm.sty
\newtheorem{thm}{Theorem}[section]
\newtheorem{pro}[thm]{Proposition}

\newtheorem{lem}[thm]{Lemma}

\theoremstyle{definition}
\newtheorem{dfn}[thm]{Definition}
\newtheorem{ntn}[thm]{Notations}
\newtheorem{ex}[thm]{Example}

\newtheorem{rem}[thm]{Remark}

\newtheorem{question}[thm]{Question}
\usepackage{lineno}
\numberwithin{equation}{section}

\def \c {\mathbb{C}}

\def \q {\mathbb{Q}}
\def \k {\mathbb{K}}
\def \s {\mathbb{S}}

\def \rank {\operatorname{rank}}

\def \cha {\operatorname{char}}

\def \det {\operatorname{det}}
\def \per {\operatorname{perm}}
\def \rk {\operatorname{\mathbf{R}}}
\def \brk {\operatorname{\underline{\mathbf{R}}}}
\def \id {\operatorname{id}}
\def \sgn {\operatorname{sgn}}

\def \SL {\operatorname{SL}}

\begin{document}

\bibliographystyle{plain}

\title{Recursive Koszul Flattenings of Determinant and Permanent Tensors}

\author{Jong In Han}
\address{%Jong In Han, 
School of Mathematics, Korea Institute for Advanced Study (KIAS), 85 Hoegi-ro, Dongdaemun-gu, Seoul 02455, Republic of Korea}
\email{jihan09@kias.re.kr}

\author{Jeong-Hoon Ju}
\address{%Jeong-Hoon Ju, 
Department of Mathematics, Pusan National University, 
	2 Busandaehak-ro 63beon-gil, Geumjeung-gu, 46241 Busan, Republic of Korea}
\email{jjh793012@naver.com}

\author{Yeongrak Kim}
\address{%Yeongrak Kim, 
Department of Mathematics \& Institute of Mathematical Sciences, Pusan National University, 
	2 Busandaehak-ro 63beon-gil, Geumjeung-gu, 46241 Busan, Republic of Korea}
\email{yeongrak.kim@pusan.ac.kr}

\thanks{}

\begin{abstract}
	We investigate new lower bounds on the tensor rank of the determinant and the permanent tensors via recursive usage of the Koszul flattening method introduced by Landsberg-Ottaviani and Hauenstein-Oeding-Ottaviani-Sommese. Our lower bounds on $\mathbf{R} (\det_n)$ completely separate the determinant and the permanent tensors by their tensor ranks. Furthermore, we determine the exact tensor ranks $\mathbf{R} (\det_4) = 12$ and $\mathbf{R} (\per_4) = 8$ over arbitrary field of characteristic $\neq 2$.
	
\end{abstract}

\keywords{tensor rank, Koszul flattening, recursive Koszul flattening,  determinant, permanent}

\subjclass[2020] {Primary 14N07, Secondary 15A15, 15A69}

\maketitle
%

%% main text

% Introduction
\section{Introduction}

The tensor rank $\mathbf{R} (T)$ of a tensor $T$ denotes the smallest number $r$ such that $T$ can be expressed as a sum of $r$ decomposable tensors. This notion generalizes the rank of a matrix, which is a fundamental concept to measure the complexity of a given matrix from algebraic and geometric points of view. One significant difference from classical linear algebra is that determining the rank of a given tensor is very challenging in most cases. For instance, the famous Strassen's algorithm tells us that the tensor rank of a matrix multiplication tensor $M_{\langle 2 \rangle}$ for $2 \times 2$ matrices is at most $7$ \cite{MR248973} (indeed, both the tensor rank and the border rank of $M_{\langle 2 \rangle}$ are exactly $7$ \cite{MR297115, MR2188132, MR3171099}), and hence provides a faster matrix multiplication algorithm for large matrices. His observation stimulates the development of complexity theory, and it is clear that the tensor rank plays a significant role in there, see \cite{MR2865915, MR3729273} for instance.

Among various tensors appearing in linear and multilinear algebra, the determinant and the permanent are particularly interesting thanks to their various applications in algebra, geometry, and combinatorics. Comparing the determinant and the permanent is a long-standing central question in complexity theory. One of the most famous questions is Valiant's conjecture which is an algebraic analogue of the $P \neq NP$ conjecture \cite{MR564634}.  Recall that the determinantal complexity of a polynomial $f$ is the minimal size of an affine linear matrix $M$ so that $f = \det (M)$. Roughly speaking, the conjecture predicts that the determinantal complexity of the permanent polynomial of degree $n$ grows faster than any polynomial in $n$. Note that when $f$ is homogeneous, there is another notion of complexity that comes from determinantal representations, called the Ulrich complexity, which denotes the minimal exponent $r$ such that $f^r$ can be expressed as the determinant of a homogeneous linear matrix \cite{MR3724216}. These questions expect that the determinant and the permanent polynomials can be distinguished by various notions in complexity, and hence we often call them ``determinant versus permanent'' problems. We refer to \cite{MR1440179,MR1861288, MR2861717} for more details on such problems together with their roles in complexity theory.

Note that Valiant's conjecture and similar $VP \neq VNP$ type questions compare the determinant and permanent of an $n \times n$ matrix in terms of polynomials of degree $n$ in $n^2$ variables. On the other hand, the main objects in this paper are the determinant and permanent tensors of order $n$, and we focus on comparing their tensor ranks which is also a mysterious problem. Let $V \cong \mathbb{K}^n$ be an $n$-dimensional vector space over a field $\mathbb{K}$, and let $\{e_1, \cdots, e_n \}$ be the standard basis for $V$. Recall that the determinant tensor is defined as the signed sum of the tensor products of permutations of basis vectors 
\begin{equation*}
	\det_n=\sum_{\sigma \in \mathfrak{S}_n} sgn(\sigma) ~ e_{\sigma(1)} \otimes e_{\sigma(2)} \otimes \cdots \otimes e_{\sigma(n)},
\end{equation*}
and the permanent tensor is defined as the unsigned sum
\begin{equation*}
	\per_n=\sum_{\sigma \in \mathfrak{S}_n}e_{\sigma(1)} \otimes e_{\sigma(2)} \otimes \cdots \otimes e_{\sigma(n)}
\end{equation*}
where $\mathfrak{S}_n$ denotes the symmetric group on $n$ letters. There are several attempts at the study of tensor ranks of $\det_n$ and $\per_n$, particularly focusing on lower bounds of $\mathbf{R} (\det_n)$ and $\mathbf{R} (\per_n)$ \cite{MR3494510, MR4284788, MR4822414}. 

The main tool we use in this paper is a generalization of the Koszul flattening, introduced by Hauenstein-Oeding-Ottaviani-Sommese \cite{MR3987862} to study the identifiability of a generic tensor, which we call the \emph{recursive Koszul flattening} in this paper. The original Koszul flattening method suggested by Landsberg and Ottaviani \cite{MR3376667} is particularly effective for tensors of order $3$, whereas the recursive Koszul flattening method aims for tensors of order $4$ or higher. Applying this method for $\det_n$ and $\per_n$, we obtain much stronger lower bounds of $\mathbf{R} (\det_n)$ for every $n$ and $\mathbf{R} (\per_n)$ for $n \le 7$. The following tables exhibit recent achievements on lower bounds of $\mathbf{R} (\det_n)$ and $\mathbf{R} (\per_n)$ over a field $\mathbb{K}$ of characteristic $0$.

\begin{table}[h!]
	\centering
	\begin{tabular}{ |c|c|c|c|c|c|c|c|c| } 
		\hline 
		$n$   & 4 & 5 & 6 & 7  & $\cdots$ \\ 
		\hline 
		Derksen \cite{MR3494510}  & 6 & 10 & 20 & 35 & $\cdots$\\ 
		Houston et al. \cite{MR4822414} & 7 & 11 & 21 & 36 & $\cdots$\\ 
		Derksen-Teitler \cite{MR3390032} &  7 &  12 &  21 & 40 & $\cdots$\\ 
		Krishna-Makam \cite{MR4284788} & -  & 17 & - & 62 & -\\
		Theorem \ref{thm:DetLower} & 11 & 27 & 65 & 164 & $\cdots$\\ 
		\hline
	\end{tabular}
	\vspace{5pt}
	\caption{Comparison of lower bounds for tensor rank of $\operatorname{det}_n$}
	\label{table1}
\end{table}

\begin{table}[h!]
	\centering
	\begin{tabular}{ |c|c|c|c|c|c|c|c|c| } 
		\hline 
		$n$   & 4 & 5 & 6 & 7  & $\cdots$ \\ 
		\hline 
		Shafiei \cite{MR3316987} & 5 & 8 & 15 & 27 & $\cdots$\\ 
		Derksen \cite{MR3494510} & 6 & 10 & 20 & 35 & $\cdots$\\ 
		Houston et al. \cite{MR4822414} & 7 & 11 & 21 & 36 & $\cdots$\\ 
		Krishna-Makam \cite{MR4284788} & - & 13 & - & 42 & -\\
		Theorem \ref{thm:PermLower} & 8 & 15 & 29 & 55 & -\\ 
		\hline
	\end{tabular}
	\vspace{5pt}
	\caption{Comparison of lower bounds for tensor rank of $\operatorname{perm}_n$}
	\label{table1}
\end{table}

In particular, we show that the border rank of $\det_n$ is at least $\frac{n^{n-1}}{(n-1)!}$ (Theorem \ref{thm:DetLower}) using the $\operatorname{SL}_n$-representation theory. From this observation, we see that $\det_n$ and $\per_n$ can be completely separated by their tensor ranks
\[
\mathbf{R} (\per_n) \le 2^{n-1} < \frac{n^{n-1}}{(n-1)!} \le \underline{\mathbf{R}} (\det_n) \le \mathbf{R} (\det_n)
\]
for every $n \ge 3$, where the leftmost inequality is from \cite{MR2673027}. Moreover, we also verify that the exact tensor rank of $\det_4$ is $12$ (Theorem \ref{thm:TensorRankDet4}) as a combination of the recursive Koszul flattening method and an explicit formula for $\det_4$ \cite[Theorem 4.1]{MR4847254} which gives an upper bound on $\mathbf{R}(\det_4)$.

The structure of this paper is as follows. In Section \ref{Sect:Preliminaries}, we recall definitions for the tensor rank and the border rank of a tensor, and backgrounds on representation theory. In Section \ref{Sect:Recursive Koszul Flattening}, we review our key ingredient, the recursive Koszul flattening method suggested by Hauenstein-Oeding-Ottaviani-Sommese. In Section \ref{Sect:Determinant Tensors}, we give a new lower bound of $\underline{\mathbf{R}}(\operatorname{det}_n)$ for every $n$. Furthermore, we determine the exact tensor rank of $\det_4$ as $\mathbf{R}(\operatorname{det}_4)=12$ over an arbitrary field of characteristic $\neq 2$. In Section \ref{Sect:Permanent Tensors}, we address experimental results on lower bounds of $\brk(\per_n)$ for $n \le 7$. In particular, we show that $\underline{\mathbf{R}}(\per_4)=\mathbf{R}(\per_4)=8$ over an arbitrary field of characteristic $\neq 2$.

%--------------------------------------------------------------

%--------------------------------------------------------------
% Acknowledgement
\section*{Acknowledgement}
J.I. H. was partially supported by the National Research Foundation of Korea (NRF) grant funded by the Korea government (MSIT) (No. RS-2024-00352592) when he was at Korea Advanced Institute of Science and Technology (KAIST). J.I. H. is currently supported by a KIAS Individual Grant (MG101401) at Korea Institute for Advanced Study.
J.-H. J. and Y. K. are supported by the Basic Science Program of the NRF of Korea (NRF-2022R1C1C1010052) and by the Basic Research Laboratory (grant MSIT no. RS-202400414849). 
The authors thank Luke Oeding and Joseph M. Landsberg for helpful comments and discussion. J.-H. J. participated the introductory school of AGATES in Warsaw (Poland) and thanks the organizers for providing a good research environment throughout the school. Y. K. thanks Eberhard Karls Universit\"at T\"ubingen and Universit\"at Konstanz for their kind hospitality during his visit. He also thanks Daniele Agostini, Kangjin Han, Yoosik Kim, and Mateusz Michalek for helpful discussion and suggestions.

%--------------------------------------------------------------

\section{Preliminaries}\label{Sect:Preliminaries}

In this section, we introduce notations and review some definitions and properties that we need throughout the paper, mostly following \cite{MR2865915, MR1153249}.

\begin{ntn}
	Throughout this paper, we use the following notations:
	\begin{itemize}
		\item $\mathbb{K}$ : a base field;
		\item $\operatorname{char}(\mathbb{K})$ : the characteristic of $\mathbb{K}$;
		\item $V, V_i,W, W_i$ : finite dimensional vector spaces over $\mathbb{K}$;
		\item $V^*$ : the dual vector space of $V$;
		\item $[d] = \{1, 2, \cdots, d \}$ where $d$ is a positive integer;
		\item $\mathfrak{S}_d$ : the symmetric group on $d$ letters;
		\item $\mathfrak{A}_d$ : the alternating group on $d$ letters;
		\item $\operatorname{GL(V)}(\text{or}~\operatorname{GL_n})$ : general linear group of $V$ ($\dim V=n$);
		\item $\operatorname{SL(V)}(\text{or}~\operatorname{SL_n})$ : special linear group of $V$ ($\dim V=n$).
	\end{itemize}
\end{ntn}

\subsection{Tensor Rank}\label{SubSect:Tensor Rank}

We first review some basic definitions of tensors and their ranks.

\begin{dfn}[Tensor]
	A map $\varphi:V_1^* \times V_2^* \times \cdots \times V_{d-1}^* \rightarrow V_d$ is said to be \emph{multilinear} if it is linear with respect to each vector space $V_i^*$ for $i \in [d-1]$. The space of such multilinear maps is identical to $V_1 \otimes V_2 \otimes \cdots \otimes V_d$. An element $T \in V_1 \otimes V_2 \otimes \cdots \otimes V_d$ is called a \emph{tensor}, and the number $d$ is called the \emph{order of $T$}.
\end{dfn}

\begin{dfn}[Tensor rank]
	Let $T \in V_1 \otimes V_2 \otimes \cdots \otimes V_d$. Then 
	\begin{equation*}
		r=\min\left\{ k ~\middle|~T=\sum_{i=1}^k v_{i,1} \otimes v_{i,2} \otimes \cdots \otimes v_{i,d} \text{ where } v_{i,j}\in V_{j} \text{ for each } j \in [d] \right\}
	\end{equation*}
	is called the \emph{tensor rank of $T$}, and denoted by $\mathbf{R}(T)$. A tensor of rank $\le 1$ is said to be \emph{decomposable}. 
\end{dfn}

\begin{dfn}[Border rank]
	Let $Z_r$ be the subspace of $V_1 \otimes V_2 \otimes \cdots \otimes V_d$ consisting of tensors of tensor rank $\leq r$, and let $\overline{Z_r}$ denote the Zariski closure of $Z_r$. Let $T \in V_1 \otimes V_2 \otimes \cdots \otimes V_d$. The integer
	\begin{equation*}
		r=\min\{k~|~T \in \overline{Z_k}\}
	\end{equation*}
	is called the \emph{border rank of $T$}, and denoted by $\underline{\mathbf{R}}(T)$. When $\mathbb{K}$ is a subfield of $\mathbb{C}$ so that it is possible to use Euclidean topology, equivalently, the border rank of $T$ is 
	\begin{equation*}
		r=\min\left\{k~|~ T=\lim_{\epsilon \rightarrow 0}T_{\epsilon}~\text{for some}~T_{\epsilon} \in V_1 \otimes V_2 \otimes \cdots \otimes V_d~\text{with} ~\mathbf{R}(T_{\epsilon})=k \right\}.
	\end{equation*}
\end{dfn}

\subsection{Representation Theory}\label{SubSect:Representation Theory}

Since the central objects of this paper are the determinant and the permanent tensors which have a lot of symmetries, we review some basic facts on representation theory. Assume that $\mathbb{K}=\mathbb{C}$ until the end of this section. 

%Briefly, representation theory is for investigating vector spaces that are invariant under group actions, and linear maps between them which are invariant under the group actions. 

Let $G$ be a group. A vector space that is invariant under a $G$-action is called a \emph{representation} of $G$ (or \emph{$G$-module}). A linear map between representations of $G$ that is invariant under the $G$-action, is called a \emph{$G$-equivariant map} (or \emph{$G$-linear map}). The detailed definition is as follows.

\begin{dfn}[$G$-equivariant map]
	Let $V$ and $W$ be representations of $G$. A linear map $\varphi:V \rightarrow W$ is said to be \emph{$G$-equivariant} if $\varphi(g \cdot v)=g \cdot \varphi(v)$ for all $g \in G, v \in V$, where $\cdot$ denotes the corresponding $G$-action to each of $V$ and $W$. If a $G$-equivariant map $\varphi$ is a linear isomorphism, then $\varphi$ is called an isomorphism.
\end{dfn}

When $V$ is a representation of $G$, a subspace $V'$ of $V$ which is also a representation of $G$ is called a \emph{subrepresentation} (or \emph{$G$-submodule}) of $V$. If $V$ has no proper nonzero subrepresentation, then $V$ is said to be \emph{irreducible}. The following, known as Schur's lemma, shows an elementary and important property of $G$-equivariant maps.

\begin{lem}[Schur's lemma]
	If $V$ and $W$ are irreducible representations of $G$ and $\varphi:V \rightarrow W$ is a $G$-equivariant map, then $\varphi$ is either a zero map or an isomorphism.
\end{lem}

A \emph{partition} of a positive integer $d$ is the set of positive integers $\lambda=(\lambda_1, \lambda_2, \ldots, \lambda_k)$ such that $\lambda_1 \geq \lambda_2 \geq \cdots \geq \lambda_k>0$ and $\lambda_1+\lambda_2+\cdots+\lambda_k=d$. 

\begin{dfn}[Young tableau]
	Let $\lambda=(\lambda_1, \lambda_2, \ldots, \lambda_k)$ be a partition of $d$. The \emph{Young diagram associated with $\lambda$} is a diagram with $\lambda_i$ boxes in the $i$th row where the rows of boxes are lined up on the left. A labelled Young diagram $T_{\lambda}=(t_{j}^{i})$ with $t_{j}^{i}\in [d]$ is called a \emph{Young tableau}. Let $t^i$ and $t_j$ denote the $i$th row and $j$th column of $T_{\lambda}$, respectively.
	\begin{itemize}
		\item [(\romannumeral1)] If each element of $[d]$ appears exactly once in $T_{\lambda}$, then $T_{\lambda}$ is called a \emph{Young tableau without repetitions}.
		\item [(\romannumeral2)] A Young tableau $T_{\lambda}$ is said to be \emph{semistandard} if the entries on each row are nondecreasing and those on each column are strictly increasing.
	\end{itemize}
\end{dfn}

Assume that $G$ is a finite group. Let $\mathbb{C}[G]$ denote the group algebra of $G$, that is, a vector space over $\c$ generated by $\{e_g~|~g \in G\}$ with algebra structure induced by $e_ge_h=e_{g\cdot h}$, where $\cdot$ denotes the group operation of $G$.

\begin{dfn}[Schur's module]
	Let  $\lambda=(\lambda_1,\lambda_2, \ldots, \lambda_k)$ be a partition of $d$ and  $T_{\lambda}=(t_j^i)$ a Young tableau without repetitions. %Consider $\mathfrak{S}_d$. 
	Let $\mathfrak{S}(t^i)$ and $\mathfrak{S}(t_j)$ denote the subgroup of $\mathfrak{S}_d$ consisting of permutations on the indices appearing in $i$th row and $j$th column, respectively. Let
	\begin{equation*}
		\rho_{t^i}=\sum_{g \in \mathfrak{S}(t^i)}e_g~~\text{and}~~\rho_{t_j}=\sum_{g \in \mathfrak{S}(t_j)}sgn(g)e_g.
	\end{equation*}
	\begin{itemize}
		\item [(\romannumeral1)] The \emph{Young symmetrizer} of $T_{\lambda}$ is defined by
		\begin{equation*}
			\rho_{t^1}\rho_{t^2}\cdots \rho_{t^k}\rho_{t_1}\rho_{t_2}\cdots \rho_{t_{\lambda_1}}
		\end{equation*}
		and denoted by $\rho_{T_{\lambda}}$.

		\item [(\romannumeral2)]  Let $\mathfrak{S}_d$ act  on $V^{\otimes d}$ on the right as
		\begin{equation*}
			(v_1 \otimes v_2 \otimes \cdots \otimes v_d)\cdot g =v_{g(1)} \otimes v_{g(2)} \otimes \cdots \otimes v_{g(d)}
		\end{equation*}
		for $g \in \mathfrak{S}_d$ and $v_1 \otimes v_2 \otimes \cdots \otimes v_d \in V^{\otimes d}$. The space
		\begin{equation*}
			\mathbb{S}_{T_{\lambda}}V:=\rho_{T_{\lambda}}(V^{\otimes d})
		\end{equation*}
		is called the \emph{Schur's module} (or \emph{Weyl's module}) corresponding to $T_{\lambda}$.
	\end{itemize} 
\end{dfn}
It is well-known that $\mathbb{S}_{T_{\lambda}}V$ is an irreducible representation of $\operatorname{GL(V)}$. Furthermore, for given Young tableaux without repetitions $T_{\lambda}$ and $\tilde{T}_{\lambda}$, Schur's modules $\mathbb{S}_{T_{\lambda}}V$ and $\mathbb{S}_{\tilde{T}_{\lambda}}V$ are isomorphic as $\operatorname{GL(V)}$-representations, and hence we denote simply $\mathbb{S}_{\lambda}V$ instead of $\mathbb{S}_{T_{\lambda}}V$.

\begin{ex}
	Let $T_{\lambda}$ be
	\begin{equation*}
		{\footnotesize 
			\begin{ytableau}
				1 & 2  \\
				3
		\end{ytableau}}~~,
	\end{equation*}
	a semistandard Young tableau associated with the partition $\lambda=(2,1)$ of $3$. Then
	\begin{align*}
		\rho_{T_{\lambda}}(v_1 \otimes v_2 \otimes v_3)&=(1+e_{(1~2)}-e_{(1~3)}-e_{(1~3~2)})(v_1 \otimes v_2 \otimes v_3)\\
		&=v_1 \otimes v_2 \otimes v_3+v_2 \otimes v_1 \otimes v_3-v_3 \otimes v_2 \otimes v_1-v_3 \otimes v_1 \otimes v_2,
	\end{align*} 
	and so
	\begin{equation*}
		\mathbb{S}_{\lambda}V=\{v_1 \otimes v_2 \otimes v_3+v_2 \otimes v_1 \otimes v_3-v_3 \otimes v_2 \otimes v_1-v_3 \otimes v_1 \otimes v_2~|~v_1,v_2,v_3 \in V\}.
	\end{equation*}
\end{ex}

Note that $\mathbb{S}_{\mu} V=\Lambda^m V$ for the partition $\mu=(1,1, \ldots, 1)$ of $m$. 

\begin{pro}[Pieri's formula]\label{prop:Pieri}
	Let $\lambda$ be a partition. Then 
	\begin{equation*}
		\mathbb{S}_{\lambda}V \otimes \Lambda^mV \cong \bigoplus_{\pi} \mathbb{S}_{\pi}V
	\end{equation*}
	where the direct sum is taken over all partitions $\pi$ whose Young diagrams are obtained from that of $\lambda$ by adding $m$ boxes, with no two in the same row.
\end{pro}

%--------------------------------------------------------------

\section{Recursive Koszul Flattening}\label{Sect:Recursive Koszul Flattening}

In this section, we review the recursive Koszul flattening method introduced by Hauenstein-Oeding-Ottaviani-Sommese \cite[Section 5.1]{MR3987862} which is a key ingredient of this paper. It is one of the rank methods that provide a lower bound of the tensor rank of a given tensor. Note that there are various rank methods \cite{MR3320240, MR3781583, MR3987583, MR4284788}, and these methods are  obtained by the subadditivity of the matrix rank. 

% Rank method

\begin{pro}[Rank method]\label{RankMethod}
	Let $\operatorname{Mat}_{l \times m}$ be the space of all $l \times m$ matrices over $\mathbb{K}$. Let
	\begin{equation*}
		\phi:V_1 \otimes \cdots \otimes V_n \rightarrow  \operatorname{Mat}_{l \times m}
	\end{equation*}
	be a linear map. Suppose that for each rank one tensor $S \in V_1 \otimes \cdots \otimes V_n$, we have
	\begin{equation*}
		\operatorname{rank}(\phi(S))\leq r.
	\end{equation*}
	Then for any tensor $T \in V_1 \otimes \cdots \otimes V_n$, we have
	\begin{equation*}
		\operatorname{rank}(\phi(T)) \leq r \cdot \underline{\mathbf{R}}(T). 
	\end{equation*}
\end{pro}

The rank method  gives a framework to improve the lower bounds on border rank, and thus those of tensor rank. We call the following method the recursive Koszul flattening method which is one of the rank methods that generalize the flattening \cite[Section 3.4]{MR2865915} and the Koszul flattening method \cite[Theorem 2.1]{MR3376667}.

\begin{pro}[{Recursive Koszul flattening \cite{MR3987862}}]\label{MainPropRecursive}
	Let $T=\sum_{i=1}^k v_{i,1} \otimes \cdots \otimes v_{i,d} \in V_1 \otimes \cdots \otimes V_d$ where $\dim V_i=m_i$ for $i \in [d]$. Fix $1 \leq s \leq d-1$ and $p_i \in [m_i-1]$ for $i \in [s]$.  
	Define 
	\begin{equation}\label{eqRKF}
		T_{(V_{1}, \ldots, V_{s})}^{\wedge(p_{1}, \ldots, p_{s})}:  \bigotimes_{i=1}^{s} \Lambda^{p_i}V_i \otimes V_{d-1}^*  \rightarrow \bigotimes_{i=1}^{s} \Lambda^{p_i+1}V_i  \otimes \bigotimes_{j=s+1}^{d-2} V_{j} \otimes   V_d
	\end{equation}
	by
	\begin{equation}\label{eqRKFdefined}
		\begin{aligned}
			& (a_{1,1} \wedge \cdots \wedge a_{1,p_1}) \otimes \cdots \otimes (a_{s,1} \wedge \cdots \wedge a_{s,p_s}) \otimes a_{d-1}^*  \\
			&\mapsto \sum_{i=1}^k 
			\left(
			\begin{aligned}
				&(v_{i,1}\wedge (a_{1,1} \wedge \cdots \wedge a_{1,p_1})) \otimes \cdots \otimes (v_{i,s} \wedge (a_{s,1} \wedge \cdots \wedge a_{s,p_s}))\\
				&\quad\quad\quad\otimes v_{i,s+1}\otimes\cdots\otimes v_{i,d-2}
				\otimes  a_{d-1}^*(v_{i,d-1}) \cdot  v_{i,d}
			\end{aligned}
			\right)
		\end{aligned}
	\end{equation}
	and extending linearly. Then
	\begin{equation}\label{eqRKFresult}
		\left\lceil \frac{\operatorname{rank}\left(T_{(V_{1}, \ldots, V_{s})}^{\wedge(p_{1}, \ldots, p_{s})}\right)}{{m_1-1 \choose p_1}{m_2-1 \choose p_2} \cdots {m_s-1 \choose p_s} }\right\rceil \leq \underline{\mathbf{R}}(T).
	\end{equation}
	We call the linear map $T_{(V_{1}, \ldots, V_{s})}^{\wedge(p_{1}, \ldots, p_{s})}$ (or, its matrix representation) a recursive Koszul flattening of $T$.
\end{pro}

It is easy to verify that
	\begin{equation*}
		\operatorname{rank}(S_{(V_{1}, \ldots, V_{s})}^{\wedge(p_{1}, \ldots, p_{s})})={m_1-1 \choose p_1}{m_2-1 \choose p_2} \cdots {m_s-1 \choose p_s} 
	\end{equation*}
for any rank one tensor $S \in V_1 \otimes \cdots \otimes V_d$, and so (\ref{eqRKFresult}) holds by the rank method. We refer to  \cite[Section 5.1]{MR3987862} and  \cite[Section 8.2]{MR3729273} for more details.

Let us briefly explain the reason why we call $T_{(V_{1}, \ldots, V_{s})}^{\wedge(p_{1}, \ldots, p_{s})}$ a recursive Koszul flattening of $T$. Consider the given tensor $T$ as a linear map by a flattening:
\begin{equation}\label{eqflatteningT}
		T \in (V_{d-1}^*)^* \otimes \left( \bigotimes_{j=1}^{d-2} V_j \otimes V_d\right).
\end{equation}
Then the canonical map $\Lambda^{p_1}V_1 \rightarrow \Lambda^{p_1+1}V_1$ applied to $T \otimes id_{\Lambda^{p_1}V_1}$ gives us a linear map called Koszul flattening of $T$:
\begin{equation*}
		T_{V_1}^{\wedge p_1} \in \left( \Lambda^{p_1}V_1 \otimes V_{d-1}^* \right)^* \otimes \left(\Lambda^{p_1+1}V_1 \otimes \bigotimes_{j=2}^{d-2} V_j \otimes V_d\right)
\end{equation*}
defined by (\ref{eqRKFdefined}) for $s=1$ and extending linearly. As we beginned with the linear map $T$ at (\ref{eqflatteningT}), if we start from $T_{V_1}^{\wedge p_1}$, then we obtain an additional linear map
\begin{equation*}
	(T_{V_1}^{\wedge p_1})_{V_2}^{\wedge p_2} \in \left(\Lambda^{p_1}V_1 \otimes \Lambda^{p_2}V_2 \otimes V_{d-1}^* \right)^* \otimes \left(\Lambda^{p_1+1}V_1 \otimes \Lambda^{p_2+1}V_2 \otimes \bigotimes_{j=3}^{d-2} V_j \otimes V_d\right)
\end{equation*}
defined by (\ref{eqRKFdefined}) for $s=2$ and extending linearly. Such a recursive process gives us (\ref{eqRKF}), which is the reason why we call (\ref{eqRKF}) a recursive Koszul flattening. 

Note that we may take an arbitrary nonempty subset $I=\{i_1, \ldots, i_s\}$ of $[d-2]$ instead of $[s]$ in order to obtain a linear map
\begin{equation*}
	T_{(V_{i_1}, \ldots, V_{i_s})}^{\wedge(p_{i_1}, \ldots, p_{i_s})}: \bigotimes_{t=1}^{s} \Lambda^{p_{i_t}}V_{i_t} \otimes V_{d-1}^* \rightarrow \bigotimes_{t=1}^{s} \Lambda^{p_{i_t+1}}V_{i_t} \otimes \bigotimes_{j \in I^c} V_{j}\otimes V_d
\end{equation*}
where $I^c$ denotes the complement of $I$ in $[d-2]$. Moreover, we may also take another flattening
\begin{equation*}
	V_l^* \rightarrow \bigotimes_{j \in [d] \setminus \{l\}} V_j
\end{equation*}
of $T$ instead of (\ref{eqflatteningT}). We regard them also as in (\ref{eqRKF}) even though these choices may permute the components in the recursive Koszul flattening.

%However, since tensor product is commutative (up to isomorphism), we write as in (\ref{eqRKF}) without loss of generality. 

Note that our central objects $\det_n$ and $\per_n$ are defined as tensors in $V^{\otimes n}$, so we have $V_1 = V_2 = \cdots = V_n = V$. We do not distinguish the subindices $V_i$ in these cases, and denote simply $(\det_n)_V^{\wedge (p_1, \ldots, p_s)}$ and $(\per_n)_V^{\wedge (p_1, \ldots, p_s)}$ by the corresponding recursive Koszul flattening of $\det_n$ and $\per_n$.

%To obtain a better lower bound of a given tensor $T$ via a rank method, we need to design a linear map $\phi$ which sends a decomposable tensor $S$ to a matrix $\phi(S)$ of low rank whereas $\phi(T)$ is a matrix of high rank. 

%The recursive Koszul flattening method is often more efficient than the original flattening, especially for tensors of order $\geq 4$. Intuitively, it is because the recursive Koszul flattening ``flattens well". For instance, consider $\operatorname{det}_4 \in V_1 \otimes V_2 \otimes V_3 \otimes V_4$. In terms of matrix, the original Koszul flattening $(\det_4)_{V_1}^{\wedge 2}$ is a $24 \times 64$ matrix. This matrix is rectangular and has rank at most $24$, and so it gives us at most $8 \leq \underline{\mathbf{R}}(\det_4)$. However, the recursive Koszul flattening $(\det_4)_{(V_1,V_2)}^{\wedge (2,1)}$ is a $96 \times 96$ matrix. This matrix is square and has rank at most $96$, and so it gives us at most $11 \leq \underline{\mathbf{R}}(\det_4)$.

%--------------------------------------------------------------

\section{Determinant Tensors}\label{Sect:Determinant Tensors}
In this section, we apply the recursive Koszul flattening on $\det_n$ to obtain a lower bound of $\underline{\mathbf{R}} (\det_n)$. An important property of $\det_n$ is that it is invariant under the standard $\SL_n$-action. We observe that the $\SL_n$-representation theory is compatible with the recursive Koszul flattening on $\det_n$, and hence it provides an improved lower bound of $\underline{\mathbf{R}}(\det_n)$.

Inspired by some computational experiments, we found that the recursive Koszul flattening works well when $(p_1,p_2,\cdots,p_{n-2})=(1,2,\cdots,n-2)$ for the determinant tensors. The role of these numbers may be found in the proof of \Cref{thm:DetLower}.
Let $V=\k^n$ and $\{e_1,e_2,\cdots,e_n\}$ be an ordered standard basis.
From now on, we fix the basis of $\Lambda^1 V\otimes \Lambda^2 V \otimes \cdots\otimes \Lambda^{n-2}V\otimes V^*$ as
\[
\mathcal{B}_1=\left\{e_{i_{1,1}}\otimes (e_{i_{2,1}}\wedge e_{i_{2,2}})\otimes \cdots \otimes (e_{i_{n-2,1}}\wedge \cdots\wedge e_{i_{n-2,n-2}})\otimes e_{i_{n-1}}^*\mid (*)\right\}
\]
where $(*)=\left(1\leq i_{j,1}<i_{j,2}<\cdots<i_{j,j}\leq n\text{ for }j\in [n-2]\text{ and }1\leq i_{n-1}\leq n\right)$.
Then we give an order on $\mathcal{B}_1$ as the one induced from the lexicographic order on $N$-tuples 
$(i_{1,1},i_{2,1},i_{2,2},\cdots,i_{n-2,1},\cdots,i_{n-2,n-2},i_{n-1})$ where $N=1+2+\cdots+(n-2)+1$. In other words, $e_1\otimes (e_1\wedge e_2) \otimes \cdots\otimes (e_1\wedge\cdots \wedge e_{n-2})\otimes e_1^*$ is the first element of $\mathcal{B}_1$, $e_1\otimes (e_1\wedge e_2) \otimes \cdots\otimes (e_1\wedge\cdots \wedge e_{n-2})\otimes e_2^*$ is the second element, and so on.
Similarly, we fix an ordered basis $\mathcal{B}_2$ of $\Lambda^2 V\otimes \Lambda^3 V \otimes \cdots\otimes \Lambda^{n-1}V\otimes V$ as the one induced from the lexicographic order on the indices. All our code implementations also follow this convention.

Meanwhile, for $1\leq k\leq n-1$, let \[\rho_{k}: \Lambda^{k+1}V \otimes \Lambda^k V \to \Lambda^k V \otimes \Lambda^{k+1}V\] be the composition of the contraction and the multiplication maps $\Lambda^{k+1}V \otimes \Lambda^k V \to (\Lambda^{k}V \otimes V) \otimes \Lambda^k V = \Lambda^{k} V \otimes (V \otimes \Lambda^{k} V) \to \Lambda^{k} V \otimes \Lambda^{k+1} V$.

\begin{lem}\label{lem:rhoIsom}
	The map $\rho_{k}$ is an isomorphism when $\k=\c$.
\end{lem}
\begin{proof}
	Note that both of $\Lambda^{k+1} V \otimes \Lambda^kV$ and $\Lambda^k V \otimes \Lambda^{k+1}V$ are representations of $\SL_n$ and $\rho_{k}$ is an $\SL_n$-equivariant map.
	Thanks to Pieri's formula (\Cref{prop:Pieri}), the representation $\Lambda^{k+1} V \otimes \Lambda^k V$ can be decomposed into Schur's modules $\s_\lambda V$ corresponding to the Young diagrams obtained by adding $k$ boxes to the $k+1$ boxes in a column such that no two boxes are added in the same row.
	Such Young diagrams are of the following form where $0\leq m\leq k$.
	
	\begin{figure}[htb]
		\centering
		\begin{tikzpicture}[
			BC/.style = {decorate, % for fancy calligraphic braces
				decoration={calligraphic brace, amplitude=5pt, raise=1mm},
				very thick, pen colour={black}
			},
			]
			\matrix (m) [matrix of math nodes,
			nodes={draw, minimum size=6mm, anchor=center},
			column sep=-\pgflinewidth,
			row sep=-\pgflinewidth
			]
			{
				~ & ~ \\
				|[draw=none,text height=3mm]| \vdots & |[draw=none,text height=3mm]| \vdots \\
				~ & ~ \\
				~ &|[draw=none,text height=3mm]| \\
				|[draw=none,text height=3mm]| \vdots \\
				~ &|[draw=none,text height=3mm]| \\
			};
			\draw[BC] (m-6-1.south west) -- node[left =2.2mm] {$2k-m+1$} (m-1-1.north west);
			\draw[BC] (m-1-2.north east) -- node[right=2.2mm] {$m$} (m-3-2.south east);
		\end{tikzpicture}
		\hspace{36pt}
	\end{figure}
	This diagram corresponds to the partition $$\lambda=(\underbrace{2, \cdots, 2}_{m}, \underbrace{1, \cdots, 1}_{2k-2m+1})$$ of $2k+1$. %where the entry 2 appears $m$ times.
	We denote the set of such partitions by $I$.
	Embed $\s_\lambda V$ into $\Lambda^{2k-m+1}V\otimes \Lambda^{m}V$ and pick $v=(e_1\wedge e_2\wedge \cdots \wedge e_{2k-m+1}) \otimes (e_1 \wedge e_2 \wedge \cdots \wedge e_m)$ which is the highest weight vector of $\s_\lambda V$.
	Let $W$ be the subspace of $\Lambda^kV$ spanned by $e_1\wedge \cdots \wedge e_k$ and $\pi:\Lambda^kV\to W$ be the projection.
	Let $p_1:\Lambda^{2k-m+1}V\otimes \Lambda^m V\to \Lambda^k V\otimes V^{\otimes k-m+1}\otimes \Lambda^m V$ be the map induced from the canonical map $\Lambda^{2k-m+1}V\to\Lambda^k V\otimes V^{\otimes k-m+1}$, and let $p_2:\Lambda^k V\otimes V^{\otimes k-m+1}\otimes \Lambda^m V\to \Lambda^k V\otimes \Lambda^{k+1}V$ be the map induced from the canonical map $V^{\otimes k-m+1}\otimes \Lambda^m V\to\Lambda^{k+1}V$.
	One can see $p_2\circ p_1$ factors through $\rho_k$ by reordering the contraction maps and wedge product maps.
	Hence we have the following commutative diagram for each $m$. 
	\[\begin{tikzcd}
		&\Lambda^{k+1}V\otimes \Lambda^kV\arrow[rd,"\rho_k"]&\\
		\Lambda^{2k-m+1}V\otimes\Lambda^mV \arrow[r,"p_1"] \arrow[ru] & \Lambda^{k}V\otimes V^{\otimes{k-m+1}}\otimes\Lambda^mV \arrow[r,"p_2"]\arrow[d, "{\pi \otimes id}"] & \Lambda^{k}V\otimes\Lambda^{k+1}V \arrow[d, "{\pi \otimes id} "] \\
		& W\otimes V^{\otimes{k-m+1}}\otimes\Lambda^mV \arrow[r] & W\otimes \Lambda^{k+1}V
	\end{tikzcd}\]
	Let $\mathfrak{S}_J$ be the symmetric group on the set $J=\{k+1,k+2,\cdots,2k-m+1\}$. Then the projection of $p_1(v)$ to $W \otimes V^{\otimes k-m+1} \otimes \Lambda^m V$ is $\sum_{\sigma \in \mathfrak{S}_J} \sgn(\sigma) ~ (e_1\wedge \cdots\wedge e_k) \otimes (e_{\sigma(k+1)}\otimes \cdots \otimes e_{\sigma(2k-m+1)}) \otimes (e_1\wedge \cdots \wedge e_m)$. This element maps to $(k-m+1)! ~ (e_1\wedge \cdots \wedge e_k) \otimes (e_1 \wedge \cdots \wedge e_m \wedge e_{k+1} \wedge \cdots \wedge e_{2k-m+1})$ in $W \otimes \Lambda^{k+1}V$. Therefore $p_2(p_1(v))$ is nonzero. As $p_2 \circ p_1$ factors through $\rho_k$, the restriction of $\Lambda^{2k-m+1} V \otimes \Lambda^m V \to \Lambda^{k+1} V \otimes \Lambda^k V$ onto $\s_\lambda V$ is nonzero as it sends $v\in \s_\lambda V$ to a nonzero element. In particular, $\s_\lambda V$ maps injectively into $\Lambda^{k+1} V \otimes \Lambda^k V$ by Schur's lemma. Similarly, $\rho_k$ sends $\s_\lambda V$ injectively into $\Lambda^{k}V\otimes\Lambda^{k+1}V$.
	As $\Lambda^{k}V\otimes \Lambda^{k+1}V$ also has the decomposition $\bigoplus_{\lambda\in I}\s_{\lambda}V$, the result follows.
\end{proof}

\begin{thm}\label{thm:DetLower}
	When $\cha(\k)=0$, we have $\brk(\det_n)\geq \dfrac{n^{n-1}}{(n-1)!}$.
\end{thm}
\begin{proof}
	First, we prove the bound over $\c$. Note that we have natural isomorphisms $V^* \cong \Lambda^{n-1}V$ and $\Lambda^nV \cong \c$ as $\SL_n$-representations. In addition, taking a tensor product with $\Lambda^n V \cong \c$ gives an isomorphism between $\SL_n$-representations. By rearranging the components, the recursive Koszul flattening of $\det_n$ with respect to $(1,2, \ldots, n-2)$ can be considered as $(\det_n)_V^{\wedge (1,2,\cdots,n-2)}:\Lambda^nV \otimes V \otimes \Lambda^2V \otimes \cdots \otimes \Lambda^{n-1}V \to V \otimes  \Lambda^2V\otimes \Lambda^3V\otimes \cdots  \otimes \Lambda^nV $. We claim that it is an isomorphism.
	
	Let
	\begin{align*}
		\widetilde{\rho}_{k}:\ &\Lambda^{k+1}V \otimes (V \otimes \cdots \otimes \Lambda^{k-1}V)\otimes \Lambda^kV \otimes \widehat{\Lambda^{k+1}V} \otimes (\Lambda^{k+2} V \otimes \cdots\otimes \Lambda^n V) \\ 
		&\to \Lambda^{k}V \otimes (V \otimes \cdots \otimes \Lambda^{k-1}V) \otimes \widehat{\Lambda^k V} \otimes \Lambda^{k+1}V\otimes  (\Lambda^{k+2}V \otimes \cdots \otimes \Lambda^nV)
	\end{align*}
	be the map induced from $\rho_{k}:\Lambda^{k+1}V \otimes \Lambda^k V \to \Lambda^k V \otimes \Lambda^{k+1}V$. In other words, $\widetilde{\rho}_{k}$ is the tensor product of the map $\rho_k$ and identity maps on $V \otimes \cdots \otimes \Lambda^{k-1}V$ and $\Lambda^{k+2} V \otimes \cdots\otimes \Lambda^n V$.
	%Define $\alpha: \wedge^n V \otimes (\wedge^{n-1} V \otimes \wedge^{n-2} V \otimes \cdots \otimes V) \to (V \otimes \cdots \otimes V) \otimes (\wedge^{n-1}V \otimes \wedge^{n-2}V \otimes \cdots \otimes V)$ as the map induced from the natural map $\wedge^n V \to V \otimes \cdots \otimes V$ so that $e_1 \wedge \cdots \wedge e_n \mapsto \det_n$.
	We define $\alpha : \Lambda^n V \to V \otimes \cdots \otimes V$ be the natural map such that $e_1 \wedge \cdots \wedge e_n \mapsto \det_n$, and define $\beta: (V \otimes \cdots \otimes V) \otimes V \otimes \Lambda^{2}V \otimes \cdots \otimes \Lambda^{n-1}V \to V \otimes \Lambda^{2}V \otimes \Lambda^{3}V \otimes \cdots \otimes \Lambda^n V$ as the tensor product of the identity map on $V_1$ and the wedge product maps $V_{k}\otimes \Lambda^{k-1}V\to\Lambda^{k}V$ for $k=2,\cdots,n$ where $V_i$ denotes the $i$-th factor of the domain. Then we have the following diagram.
	
	{\[\begin{tikzcd}
			\Lambda^nV\otimes V\otimes \cdots \otimes \Lambda^{n-2}V\otimes \Lambda^{n-1}V \arrow[d,swap,"\widetilde{\rho}_{n-1}"]\arrow[rdd,"
			 \alpha \otimes id"]&\\
			\Lambda^{n-1}V\otimes V\otimes \cdots \otimes \Lambda^{n-2}V\otimes \Lambda^{n}V \arrow[d,swap,"\widetilde{\rho}_{n-2}"]&\\
			\Lambda^{n-2}V\otimes V\otimes \cdots \otimes \Lambda^{n-1}V\otimes \Lambda^{n}V \arrow[d,swap,"\widetilde{\rho}_{n-3}"]& (V \otimes \cdots \otimes V) \otimes V\otimes \cdots \otimes \Lambda^{n-2}V\otimes \Lambda^{n-1}V\arrow[ldd,"\beta"]\\
			\vdots \arrow[d,swap,"\widetilde{\rho}_{1}"]&\\
			V\otimes \Lambda^{2}V\otimes \cdots \otimes \Lambda^{n-1} V\otimes \Lambda^{n} V&
		\end{tikzcd}\]}
	The commutativity can be seen by reordering the contraction maps and wedge product maps.
	By definition, $\beta \circ (\alpha \otimes \id)$ is the recursive Koszul flattening of $\det_n$. It is an isomorphism as each $\widetilde{\rho}_{k}$ is an isomorphism (\Cref{lem:rhoIsom}), in particular, 
	\[
	\brk(\det_n)\geq \frac{n\prod_{i=1}^{n-2}\binom{n}{i}}{\prod_{i=1}^{n-2}\binom{n-1}{i}}=\frac{n^{n-1}}{(n-1)!}.
	\]
	
	For a field $\k$ of characteristic zero, let $M_n(\k)$ be the matrix over $\mathbb{K}$ corresponding to the recursive Koszul flattening of $\det_n$ with respect to the bases $\mathcal{B}_1$ and $\mathcal{B}_2$.
	Then the nonzero entries of $M_n(\k)$ are either $1$ or $-1$, and even if the field changes, the positions of $1$ and $-1$ do not change. Hence, both $\det (M_n(\k))$ and $\det (M_n(\c))$ correspond to the same integer. As the latter one is shown to be nonzero, we have the same result for $\k$.
\end{proof}

\begin{rem}
	From Stirling’s approximation, we have $\frac{n^{n-1}}{(n-1)!} \approx \frac{e^n}{\sqrt{2\pi n}}$ for sufficiently large $n$. Hence, when $\cha({\k})=0$, the lower bound $\frac{n^{n-1}}{(n-1)!}$ of $\rk(\det_n)$ increases much faster than the upper bound $2^{n-1}$ of $\rk(\per_n)$ as $n$ increases. In particular, our lower bound on $\rk (\det_n)$ distinguishes $\det_n$ and $\per_n$ by their tensor ranks for every $n \ge 3$, and the $\rk (\det_n)$ increases much faster than $\rk (\per_n)$ so that the ratio are at least exponential in $n$. 
	
	Recently, Houston, Goucher, and Johnston reported that $\rk (\det_n) \le B_n$ where $B_n$ is the $n$-th Bell number \cite[Theorem 1]{MR4822414}. Thanks to the asymptotic estimation for the Bell number 
\[
	\left( \frac{n}{e \ln n} \right)^n < B_n < \left( \frac{n}{e^{1-\epsilon} \ln n} \right)^n
\]
which can be derived from de Bruijn's formula \cite{MR99564}, the best-known upper bound on $\rk (\det_n)$ at this moment is not of polynomial growth in $2^{n-1}$, the best-known upper bound on $\rk (\per_n)$, whereas the best-known lower bound in \Cref{thm:DetLower} has a polynomial growth in $2^{n-1}$. In a viewpoint of algebraic complexity theory, it is very tempting to ask the following question. 

\begin{question}
Does the tensor rank $\rk (\det_n)$ grow faster than any polynomial in $\rk (\per_n)$?
\end{question}

\end{rem}

As an immediate consequence of \Cref{thm:DetLower}, we get $\brk(\det_4)\geq 11$ and $\brk(\det_5)\geq 27$ when $\cha(\k)=0$. When the characteristic is finite, we get the following result.

\begin{thm}\label{thm:det4FiniteChar}
	When $\cha(\k)\geq 3$, it holds that $\brk(\det_4)\geq 11$, and when $\cha(\k)\geq 5$, it holds that $\brk(\det_5)\geq 27$.
\end{thm}
\begin{proof}
	We define $M_n(\k)$ similarly as in the proof of \Cref{thm:DetLower}.
	By the same reason, when $\cha(\k)=p>0$, it holds that $\det (M_n(\k))\equiv \det (M_n(\q)) \pmod p$. 
	A computational result shows that $\det (M_4(\q))=2^{32}$ when $n=4$. When $n=5$, we get $\det (M_5(\q))=2^{1600} 3^{25}$, and hence the result follows. The implementation can be found in \cite{hanjukim2024appendix}.
\end{proof}

Now we proceed to determine the exact tensor rank of $\det_4$.
The following lemma is used when Krishna-Makam showed $\rk(\det_3)=5$ over the fields of characteristic two \cite{MR4284788}. Their proof uses a fact that $\det_3$ is invariant under the $(\SL_3\times C_3)$-action where $C_3$ denotes the cyclic group of order $3$ so that any simple tensor can be written simply. We use a similar argument to show $\rk(\det_4)=12$ together with the recursive Koszul flattening.

\begin{lem}[{\cite[Lemma 3.6]{MR4284788}}]\label{lem:minusRankOne}
	Let $T\in V^{\otimes n}$ be a tensor such that $\rk(T-S)\geq r$ for all tensors $S\in V^{\otimes n}$ of rank one. Then $\rk(T)\geq r+1$.
\end{lem}
\begin{proof}
	If $T=S_1+\cdots+S_r$ for some rank one tensors $S_1,\cdots,S_r$, then we get a contradiction as $\rk(T-S_1)\leq r-1$.
\end{proof}

\begin{thm}\label{thm:TensorRankDet4}
	When $\cha(\k) \neq 2$, we have $\rk(\det_4)=12$.
\end{thm}
\begin{proof}
	Assume first that $\cha(\k)=0$. 
	By \Cref{lem:minusRankOne}, it is enough to show $\rk(\det_4-S)\geq 11$ for every tensor $S=s_1 \otimes s_2 \otimes s_3 \otimes s_4$ of rank one. We denote the $96 \times 96$ matrix corresponding to the recursive Koszul flattening of $\det_4-S$ with respect to $(1,2)$ by $M$. As $\det_4$ is invariant under the natural $(\SL_4\times \mathfrak{A}_4)$-action, we divide into four cases.
	
	Suppose $\dim \langle s_1,s_2,s_3,s_4 \rangle=4$. Then there is an element $(g,1) \in \SL_4\times \mathfrak{A}_4$ such that $(g,1)\cdot S=x(e_1\otimes e_2\otimes e_3\otimes e_4)$ for some $x\in\k$, i.e., $(g,1)\cdot(\det_4-S)=\det_4-x(e_1\otimes e_2\otimes e_3\otimes e_4)$. As the $(\SL_4\times \mathfrak{A}_4)$-action preserves the rank, we may assume $S=x (e_1\otimes e_2\otimes e_3\otimes e_4)$.
	In this case, we observe that $\det (M)=-2^{20}(x-1)(x-2)^4(x-4)^4$. Therefore, the matrix $M$ has rank $96$ so that $\brk(\det_4-S)\geq \frac{96}{9}$ when $x\neq 1,2,4$. When $x=1$, the rank of $M$ drops to $\rank (M)=95$, hence $\brk(\det_4-S) \geq \frac{95}{9}$. Similarly, when $x=2$, we get $\brk(\det_4-S)\geq \frac{93}{9}$ as $\rank (M)=93$, and when $x=4$, we get $\brk(\det_4-S)\geq \frac{92}{9}$ as $\rank (M)=92$. 
	
	Suppose $\dim\langle s_1,s_2,s_3,s_4\rangle=3$. Then there is an element $(g,\sigma)\in \SL_4\times \mathfrak{A}_4$ such that $(g,\sigma)S=e_1\otimes e_2\otimes e_3\otimes(xe_1+ye_2+ze_3)$ for some $x,y,z\in \k$. Thus in this case, we may assume $S=e_1\otimes e_2\otimes e_3\otimes(xe_1+ye_2+ze_3)$. In this case, we get $\det (M)=2^{32}$. Hence $\brk(\det_4-S)\geq \frac{96}{9}$.
	
	Suppose $\dim\langle s_1,s_2,s_3,s_4\rangle=2$. In this case, we may assume $S=e_1\otimes e_2\otimes (xe_1+ye_2)\otimes (ze_1+we_2)$. As before, we get $\det (M)=2^{32}$ and $\brk(\det_4-S)\geq \frac{96}{9}$.
	
	Suppose $\dim\langle s_1,s_2,s_3,s_4\rangle=1$. In this case, we may assume $S=x(e_1\otimes e_1\otimes e_1\otimes e_1)$. As before, we get $\det (M)=2^{32}$ and $\brk(\det_4-S)\geq \frac{96}{9}$. In any cases, we conclude that $\brk (\det_4 - S) \ge 11$ which implies the bound $\mathbf{R} (\det_4) \ge 12$.
	
	Now assume $\cha (\k)=p\geq 3$. We proceed as above. When $\dim \langle s_1,s_2,s_3,s_4\rangle\leq 3$, we still have $\brk(\det_4-S)\geq \frac{96}{9}$ since $\det (M)\equiv 2^{32} \pmod p$. Hence we consider the case $\dim \langle s_1,s_2,s_3,s_4\rangle=4$ only and we may assume $S=x(e_1\otimes e_2\otimes e_3\otimes e_4)$. Then $M$ is invertible unless $x=1,2,4$ since the roots of $\det (M)$ are 1, 2, and 4 when $\cha (\k)>3$ while 1 and 2 are the only roots when $\cha (\k)=3$. We could find two $91\times 91$ submatrices $M'$ and $M''$ of $M$ so that $\det (M')\equiv 2^{20}\pmod p$ when $x=1,4$ and $\det (M'')\equiv 2^{23}\pmod p$ when $x=2$. This shows $\brk(\det_4-S)\geq \frac{91}{9}$ in any case. 
	
	It is known that $\mathbf{R} (\det_4) \le 12$ thanks to \cite[Theorem 4.1]{MR4847254} unless $\cha(\k) =2$, we have the exact value of the tensor rank $\mathbf{R} (\det_4) = 12$. We refer to \cite{hanjukim2024appendix} for the implementation of these computations.
\end{proof}

Note that over the field $\mathbb{F}_2$ of two elements, it is known that $\rk(\det_4)=12$ due to \cite[Theorem 5]{MR4822414}.
Together with \Cref{thm:TensorRankDet4}, the only case left for $\rk(\det_4)$ is over a field $\k$ of characteristic $2$ and $\k \neq \mathbb{F}_2$. We expect that $\rk(\det_4)=12$ also in this case, however, we leave it as an open question.

Due to \Cref{thm:DetLower} and \Cref{thm:det4FiniteChar}, we have $11\leq\brk(\det_4)\leq 12$ when $\cha(\k)\neq 2$. It is natural to consider the following question:

\begin{question}
	Does the border rank $\brk(\det_4)$ equal to $12$?
\end{question}

In the case of $\mathbb{K}=\mathbb{C}$, we obtain an affirmative answer using border apolarity theory developed in \cite{MR4332674, MR4595287}, and it will appear in our upcoming preprint \cite{hanjukim2025border}.

%-------------------------------------------------------------------------------------------------------

\section{Permanent Tensors}\label{Sect:Permanent Tensors}
Unlike the determinant tensor $\det_n$, the permanent tensor $\per_n$ is not $\SL_n$-invariant in general, so the purpose of this section is to address computational and experimental results on lower bounds of $\underline{\mathbf{R}} (\per_n)$ for small values of $n$.  

In practice, we may expect a better lower bound achieved by a rank method when we compare a given tensor with the rank of a square matrix. Hence, as in the case of determinant tensor, we also design our experiments so that we observe the rank of a square matrix representation of $(\per_n)^{\wedge(1,2,\cdots,n-2)}_V:\Lambda^1 V\otimes \Lambda^2 V \otimes \cdots\otimes \Lambda^{n-2}V\otimes V^* \to \Lambda^2 V\otimes \Lambda^3 V \otimes \cdots\otimes \Lambda^{n-1}V\otimes V$ although most of the contents in this section can be applied to the other recursive Koszul flattenings $(\per_n)^{\wedge(p_1,p_2,\cdots,p_{s})}_V$. 

Since the size of the matrices becomes huge when we increase $n$, we introduce some useful tools to calculate their rank; for instance, we need to calculate the rank of a $26471025 \times 26471025$ matrix when $n=7$ which is a heavy computation in practice. Fortunately, the corresponding recursive Koszul flattening map is represented by a sparse matrix, and the permanent tensor $\per_n$ is invariant under the natural $\mathfrak{S}_n$-action. This observation gives a certain symmetry on $(\per_n)^{\wedge(p_1,p_2,\cdots,p_{s})}_V$ that may reduce the computational cost a lot. 

Consider the bases $\mathcal{B}_1, \mathcal{B}_2$ we defined in the previous section. For $\sigma\in\mathfrak{S}_n$ and $x=e_{i_{1,1}}\otimes (e_{i_{2,1}}\wedge e_{i_{2,2}})\otimes \cdots \otimes (e_{i_{n-2,1}}\wedge \cdots\wedge e_{i_{n-2,n-2}})\otimes e_{i_{n-1}}^*\in\mathcal{B}_1$, we define $\sigma(x):=e_{\sigma(i_{1,1})}\otimes (e_{\sigma(i_{2,1})}\wedge e_{\sigma(i_{2,2})})\otimes \cdots \otimes (e_{\sigma(i_{n-2,1})}\wedge \cdots\wedge e_{\sigma(i_{n-2,n-2})})\otimes e_{\sigma(i_{n-1})}^*$. By spanning it linearly, we get an $\mathfrak{S}_n$-action on $\Lambda^1 V\otimes \Lambda^2 V \otimes \cdots\otimes \Lambda^{n-2}V\otimes V^*$ which is actually a representation. Note that when $x\in\mathcal{B}_1$, either $\sigma(x)$ or $-\sigma(x)$ is contained in $\mathcal{B}_1$. We denote the one contained in $\mathcal{B}_1$ as $|\sigma(x)|$.
We similarly define an $\mathfrak{S}_n$-action on $\Lambda^2 V\otimes \Lambda^3 V \otimes \cdots\otimes \Lambda^{n-1}V\otimes V$ that is a representation.

\begin{dfn}\label{DefConnectedComponent}
	Let $\phi:W_1\to W_2$ be a linear map. Fix the ordered bases $\mathcal{B}_1$ of $W_1$ and $\mathcal{B}_2$ of $W_2$, and let $M$ be the corresponding matrix.
	\begin{itemize}
		\item [(\romannumeral1)] For $a\in\mathcal{B}_1$ and $b\in\mathcal{B}_2$, we say $a$ and $b$ are \textit{adjacent} if $b$ has a nonzero coefficient in the linear combination of elements of $\mathcal{B}_2$ that represents $\phi(a)$.
		\item [(\romannumeral2)] For $x,y\in\mathcal{B}_1\cup\mathcal{B}_2$, we say $x$ and $y$ are \textit{connected} if there exists a sequence of elements $z_1,\cdots,z_n\in\mathcal{B}_1\cup\mathcal{B}_2$ such that $z_i$ is adjacent to $z_{i+1}$ for all $i$ where $z_1=x$ and $z_{n}=y$.
		\item [(\romannumeral3)] For an element $a\in \mathcal{B}_1$ and a submatrix $N$ of $M$, we say $a$ \textit{is contained in} $N$ if $N$ contains the column, or possibly just a part of the column, corresponding to $a$. We define similarly for the elements in $\mathcal{B}_2$.
		\item [(\romannumeral4)] We call a submatrix $N$ of $M$ is a \textit{connected component} if
	\begin{itemize}
		\item [$\bullet$] any pair of the rows and columns contained in $N$ are connected, and
		\item [$\bullet$] any row or column that is not contained in $N$ is not connected to any row or column contained in $N$.
	\end{itemize}
	When $a\in\mathcal{B}_1$ has an adjacent element, we denote the connected component containing $a$ as $M(a)$.
		\item [(\romannumeral5)] We say two matrices are \textit{essentially same} if one can be obtained by reordering the rows and columns and multiplying nonzero scalars to the rows and columns from the other.
	\end{itemize}
\end{dfn}

\begin{ex}
	Let $M$ be the matrix
	\[\begin{bNiceMatrix}[first-row,first-col]
		& a_1 & a_2 & a_3 & a_4 & a_5 & a_6 & a_7 \\
		b_1 & 1 & 0 & 0 & 2 & 0 & 0 & 0 \\
		b_2 & 0 & 3 & 0 & 0 & 0 & 0 & 4 \\
		b_3 & 0 & 0 & 0 & 5 & 0 & 6 & 0 \\
		b_4 & 0 & 0 & 0 & 0 & 0 & 0 & 0 \\
		b_5 & 0 & 0 & 0 & 0 & 7 & 0 & 8 \\
		b_6 & 0 & 0 & 0 & 0 & 0 & 9 & 0 \\
	\end{bNiceMatrix}\]
	correspondong to a linear map $\phi:W_1\to W_2$ with respect to ordered bases $\mathcal{B}_1=\{a_1,\cdots,a_7\}$ and $\mathcal{B}_2=\{b_1,\cdots,b_6\}$.
	Then its connected components are
	\[
	\begin{bNiceMatrix}[first-row,first-col]
		& a_1 & a_4 & a_6 \\
		b_1 & 1 & 2 & 0 \\
		b_3 & 0 & 5 & 6 \\
		b_6 & 0 & 0 & 9 \\
	\end{bNiceMatrix}
	\quad\text{ and }\quad
	\begin{bNiceMatrix}[first-row,first-col]
		& a_2 & a_5& a_7\\
		b_2 & 3 & 0 & 4 \\
		b_5 & 0 & 7 & 8 \\
	\end{bNiceMatrix}
	\]
	and the rank of $M$ is equal to the sum of the ranks of these.
	Note that $a_3$ is not adjacent to any element.
\end{ex}

\begin{lem}\label{lem:SnEquivariant}
	The recursive Koszul flattening $(\per_n)^{\wedge(1,2,\cdots,n-2)}_V$ is $\mathfrak{S}_n$-equivariant, i.e., for $\sigma\in\mathfrak{S}_n$ and $a\in \Lambda^1 V\otimes \Lambda^2 V \otimes \cdots\otimes \Lambda^{n-2}V\otimes V^*$, we have \[(\per_n)_V^{\wedge(1,2,\cdots,n-2)}(\sigma(a))=\sigma((\per_n)_V^{\wedge(1,2,\cdots,n-2)}(a)).\]
\end{lem}
\begin{proof}
	This can be verified easily by calculating the both sides.
\end{proof}

Note that $(\det_n)^{\wedge(1,2,\cdots,n-2)}_V$ is not $\mathfrak{S}_n$-equivariant but $\mathfrak{A}_n$-equivariant, and furthermore, we have $(\det_n)_V^{\wedge(1,2,\cdots,n-2)}(\sigma(a))=\sgn(\sigma)\cdot\sigma((\det_n)_V^{\wedge(1,2,\cdots,n-2)}(a))$.

\begin{lem}\label{lem:sigmaEssentiallySame}
	Let $M_n$ be the matrix corresponding to $(\per_n)^{\wedge(1,2,\cdots,n-2)}_V$ with respect to the ordered bases $\mathcal{B}_1$ and $\mathcal{B}_2$. Suppose $a\in\mathcal{B}_1$ has an adjacent element. Then for any $\sigma\in\mathfrak{S}_n$, the element $\sigma(a)$ also has an adjacent element, and the connected component $M_n(|\sigma(a)|)$ is essentially same as $M_n(a)$. Furthermore, the elements $a$ and $a'$ of $\mathcal{B}_1$ are contained in the same connected component if and only if $|\sigma(a)|$ and $|\sigma(a')|$ are contained in the same connected component.
\end{lem}
\begin{proof}
	Let $a_1,\cdots,a_l\in\mathcal{B}_1$ and $b_1,\cdots,b_m\in\mathcal{B}_2$ be the elements contained in $M(a)$. Write each $(\per_n)^{\wedge(1,2,\cdots,n-2)}_V(a_i)$ as a linear combination of the elements in $\mathcal{B}_2$. By \Cref{lem:SnEquivariant}, applying $\sigma$ to the linear combinations, we get the elements
	$\sigma(b_1),\cdots,\sigma(b_m)\in\Lambda^2 V\otimes \Lambda^3 V \otimes \cdots\otimes \Lambda^{n-1}V\otimes V$ and the linear combinations of $\sigma(b_1),\cdots,\sigma(b_m)$ that represent $(\per_n)^{\wedge(1,2,\cdots,n-2)}_V(\sigma(a_i))$ for each $i$. Those linear combinations show that $|\sigma(a_1)|,\cdots,|\sigma(a_l)|\in\mathcal{B}_1$ and $|\sigma(b_1)|,\cdots,|\sigma(b_m)|\in\mathcal{B}_2$ are all connected. Also, the other elements in $\mathcal{B}_1$ and $\mathcal{B}_2$ are not connected to any of these, since if there exists such element $x$ that is connected to, for example, $|\sigma(a_1)|$, then there is a sequence of linear combinations that connects $x$ and $|\sigma(a_1)|$, and we get a contradiction by applying $\sigma^{-1}$ to these linear combinations. Therefore $|\sigma(a_1)|,\cdots,|\sigma(a_l)|\in\mathcal{B}_1$ and $|\sigma(b_1)|,\cdots,|\sigma(b_m)|\in\mathcal{B}_2$ form the connected component $M_n(|\sigma(a)|)$, and the rows and columns of $M_n(|\sigma(a)|)$ can be obtained by reordering the rows and columns of $M_n(a)$ and multiplying $-1$ to some rows and columns of $M_n(a)$. Hence $M_n(a)$ and $M_n(|\sigma(a)|)$ are essentially the same, and the last statement follows.
\end{proof}

This lemma motivates to define an equivalence relation $\sim$ on the set $\mathcal{B}_1'=\{a\in\mathcal{B}_1\mid a\text{ has an adjacent element}\}$ as follows: for $a,a'\in\mathcal{B}_1'$, we set $a\sim a'$ if and only if either there exists $\sigma\in\mathfrak{S}_n$ such that $a$ and $|\sigma(a')|$ are contained in the same connected component of $M$, where $M$ is as in Definition \ref{DefConnectedComponent}. Considering the equivalence classes with respect to the relation $\sim$, we obtain the following lemma.

\begin{lem}\label{lem:calcRank}
	Let $M_n$ be the matrix corresponding to $(\per_n)^{\wedge(1,2,\cdots,n-2)}_V$ with respect to the ordered bases $\mathcal{B}_1$ and $\mathcal{B}_2$ and $H_a=\{\sigma\in\mathfrak{S}_n\mid |\sigma(a)|\text{ is contained in }M_n(a)\}$. Then $H_a$ is a subgroup of $\mathfrak{S}_n$ and \[\rank (M_n)=\sum_{a}\frac{n!\rank(M_n(a))}{|H_a|}\] where the summation is taken over the representatives of the equivalence classes on $\mathcal{B}_1'$ with respect to $\sim$.
\end{lem}
\begin{proof}
	Fix $a\in \mathcal{B}_1'$. For any $\sigma,\tau\in\mathfrak{S}_n$, it holds that $M_n(a)=M_n(|\sigma(a)|)$ if and only if $M_n(|\tau(a)|)=M_n(|\tau(\sigma(a))|)$ by \Cref{lem:sigmaEssentiallySame}.
	Hence for $\sigma_1,\sigma_2\in H_a$, we get $M_n(a)=M_n(|\sigma_1^{-1}\sigma_2(a)|)$ from $M_n(|\sigma_1(a)|)=M_n(|\sigma_2(a)|)$. Together with $\id\in H_a$, we get $H_a\leq \mathfrak{S}_n$ and the distinct connected components in $\{M_n(|\sigma(a)|)\}_{\sigma\in\mathfrak{S}_n}$ correspond to the left cosets of $H_a$ in $\mathfrak{S}_n$.
	By \Cref{lem:sigmaEssentiallySame}, those connected components are essentially the same as $M_n(a)$, so they have the same rank.
\end{proof}

Then we compute the rank of $(\per_n)^{\wedge(1,2,\cdots,n-2)}_V$ up to $n=7$ with this lemma and conclude the following.

\begin{thm}\label{thm:PermLower}
	When $\cha (\k)=0$, the following hold:
	\begin{itemize}
		\setlength\itemsep{0.5em}
		\item $\rk(\per_4)=\brk(\per_4)=8$;
		\item $\brk(\per_5)\geq 15$;
		\item $\brk(\per_6)\geq 29$;
		\item $\brk(\per_7)\geq 55$.
	\end{itemize}
\end{thm}
\begin{proof}
	Note that we have the upper bound $\brk(\per_4)\leq \rk(\per_4)\leq 8$ due to D. G. Glynn \cite{MR2673027}.
	We implemented an algorithm that uses \Cref{lem:calcRank} to calculate the rank of the matrix $M_n$ corresponding to $(\per_n)^{\wedge(1,2,\cdots,n-2)}_V$ with respect to the ordered bases $\mathcal{B}_1$ and $\mathcal{B}_2$ and obtained the following result.
	\[\begin{array}{c|c|c|c}
		n & \text{size of }M_n & \rank (M_n) & \text{lower bound of }\brk(\per_n)\\ \hline
		4 & 96\times 96 & 70 & 8\\
		5 & 2500\times 2500 & 1426 & 15\\
		6 & 162000\times 162000 & 70692 & 29\\
		7 & 26471025\times 26471025 & 8763494 & 55
	\end{array}\]
	The implementation can be found in \cite{hanjukim2024appendix}.
\end{proof}

\begin{thm}\label{thm:perFiniteChar}
	When $\cha(\k)\neq 2$, we have $\rk(\per_4)=\brk(\per_4)=8$ and $\brk(\per_5)\geq 15$.
\end{thm}
\begin{proof}
	Due to \Cref{thm:PermLower}, we only need to consider when $\cha(\k)=p>0$.
	Let $M_n(\mathbb{K})$ be the matrix over $\mathbb{K}$ corresponding to $(\per_n)^{\wedge(1,2,\cdots,n-2)}_V$ with respect to the ordered bases $\mathcal{B}_1$ and $\mathcal{B}_2$
	As before, it holds that $\det (M_n(\k))\equiv\det (M_n(\q))\pmod p$. 
	Also, one can check there is a $70\times 70$ submatrix of $M_4(\q)$ with the determinant 1. Hence $\det (M_4(\k))\neq 0$ for any field $\k$. Similarly, one can check there is a $1426\times 1426$ submatrix of $M_5(\q)$ with the determinant $-1$ so that $\det (M_5(\k))\neq 0$ for any field $\k$. The implementation can be found in \cite{hanjukim2024appendix}.
\end{proof}

Note that Comon's conjecture \cite{MR2447451} predicts that $\rk(\per_n)$ is the same as the Waring rank of the monomial $x_1 \cdots x_n$, which is exactly $2^{n-1}$ when $\k$ is algebraically closed with $\cha({\k})=0$ \cite[Proposition 3.1]{MR2966824}, $\k=\mathbb{R}$ \cite[Theorem 3.5]{MR3648510}, or $\k=\mathbb{Q}$ \cite[Corollary 3.24]{MR4461573}. In addition, when $\cha({\k})\neq 2$, it is proved that the prediction is true for $n=3$ by \cite[Corollary 1.14]{MR3987583} and for $n=4$ by Theorem \ref{thm:perFiniteChar}. Hence it is natural to ask the following question.

\begin{question}
	Determine whether the equality $\rk(\per_n)=2^{n-1}$ holds for each $n \ge 5$.
\end{question}

\bibliography{ref.bib}
\vspace{0.5cm}

\end{document}